\documentclass[graybox]{svmult}
\usepackage{mathptmx}       % selects Times Roman as basic font
\usepackage{helvet}         % selects Helvetica as sans-serif font
\usepackage{courier}        % selects Courier as typewriter font
\usepackage{type1cm}        % activate if the above 3 fonts are
\usepackage{amsfonts, amssymb, amsmath}
\usepackage{makeidx}         % allows index generation
\usepackage{graphicx}        % standard LaTeX graphics tool
\usepackage{multicol}        % used for the two-column index
\usepackage[bottom]{footmisc}% places footnotes at page bottom

\makeindex             % used for the subject index
\newcommand{\F}{\mathcal {F}} %% PLEASE DO NOT 
\newcommand{\RR}{\mathbb R} %% REVMOVE 
\newcommand{\NN}{\mathbb N} %% THESE 
\newcommand{\ZZ}{\mathbb Z}%% COMMANDS
\newcommand{\PP}{\mathbb{P}} %%
\newcommand{\eps}{\varepsilon}%%

\begin{document}
\title*{Almost sure rates of mixing for random intermittent maps}
\author{Marks Ruziboev}

\institute{Marks Ruziboev \at Department of Mathematical Sciences, Loughborough University, Loughborough, Leicestershire, LE11 3TU, UK
\email{M.Ruziboev@lboro.ac.uk}\\ 
Dedicated to Abdulla Azamov and Leonid Bunimovich on the occasion of their 70th birthday}

%\subjclass[2010]{37A25, 37A50 37C40  37E05 }
%\date{\today}
%\keywords{Random dynamical systems, slowly mixing systems, quenched decay of correlations.}

\maketitle

\abstract{We consider a family $\mathcal F$ of maps with two branches and a common neutral  fixed point $0$ such that the order of tangency at $0$  belongs to  some interval $[\alpha_0, \alpha_1]\subset (0, 1)$.  Maps in $\mathcal F$ do not necessarily share a common Markov partition.  At each step a member of $\mathcal F$ is chosen independently with respect to the uniform distribution on $[\alpha_0, \alpha_1]$.  We show that the construction of the random  tower in Bahsoun-Bose-Ruziboev \cite{BBR}   with \emph{general return time} can be carried out for  random compositions of such maps. Thus their  general results are applicable and gives upper bounds for the quenched decay of correlations of form $n^{1-1/\alpha_0+\delta}$ for any $\delta>0$. }

\section{Introduction}
In recent years there has been remarkable interest  in studying statistical properties of random dynamical systems induced by random compositions of different maps (see for example \cite{ANV}-\cite{BBMD}, \cite{HRY}, \cite{K}, \cite{LS}, \cite{NTV} and references therein).  In \cite{BBD}  i.i.d. random compositions of two Liverani-Saussol-Vaienti (LSV)\footnote{A subclass of the so called  Pomeau-Manneville maps introduced in \cite{PM}, and  popularised by Liverani, Saussol and Vaienti in \cite{LSV}.  Such systems have attracted attention of both mathematicians and physicists  (see \cite{LS} for a recent work in this area).} maps were considered and
it was shown that the rate of decay of the annealed (averaged over all realisations) correlations is given by the fast dynamics.  Recently  the general results on quenched decay rates (i.e. decay rates for almost every realisation) for the random compositions of non-uniformly expanding maps  were obtained in \cite{BBR}.  As an illustration it was shown ibidem that the general results are applicable to the random map induced by  compositions of LSV   maps with parameters in $[\alpha_0, \alpha_1]\subset(0, 1)$ chosen with respect to a suitable distribution $\nu$ on $[\alpha_0, \alpha_1]$. In the current note we, fix the uniform distribution on $[\alpha_0, \alpha_1]$ and consider a  family of maps with  common neutral fixed point. Our maps do \emph{not} share a common Markov partition.  We show that the construction of the random  tower of \cite{BBR}   with \emph{general return time} can be carried out for the random compositions of such maps. Hence the main result of \cite{BBR} is applicable. We obtain upper bounds for the quenched decay of correlations of the form $n^{1-1/\alpha_0+\delta}$ for any $\delta>0$.  

The paper is organised as follows. In Section \ref{setup}, we give  a formal definition of the family $\F$ and state the main result of the paper (Theorem \ref{main}). In Section \ref{induce},  we construct uniformly expanding  induced  random map and show  that the assumptions required in \cite{BBR}  are satisfied, i.e. we check  uniform expansion, bounded distortion, decay rates for the tail of the return time and aperiodicity. Also we formulate a  technical proposition in this section which is used to obtain the tail estimates and proved in  Section \ref{pof:key}.

\section{The set up and the main results}\label{setup}
In this section we define the main object of the current note: the random maps.
Fix two real numbers $0<\alpha_0<\alpha_1<1$. Let $I=[0, 1]$ and let $\F$ be a parametrised family of maps  $T_\alpha: I\to I$, $\alpha\in [\alpha_0, \alpha_1]$  with the following properties. 
\begin{itemize}
\item[(A1)] { $ \quad$ There exists a $C^1$ function $x:[\alpha_0, \alpha_1]\to (0, 1)$,  $\alpha \mapsto x_\alpha$   such that $T_\alpha: [0, x_\alpha)\to  [0, 1)$ and $T_\alpha: [x_\alpha, 1]\to [0, 1]$ are  increasing diffeomorphisms.}  
\item[(A2)] { $ \quad$$T_\alpha'(x)>1$ for any $x>0$. }
\item[(A3)]  $ \quad$There exists $\eps_0>0$ and continuous functions $\alpha\mapsto c_\alpha$,  $(x, \alpha)\mapsto f_\alpha(x)$ such that $f_\alpha(0)=0$ and $T_\alpha(x)=x+c_\alpha x^{1+\alpha}(1+f_\alpha(x))$ for any $x\in [0, \eps_0]$.
\item[(A4)]  $ \quad$Every  $T_\alpha$ is $C^3$ on $(0, x_\alpha]$ with negative Schwarzian derivative.
\item[(A5)] $ \quad$ $(x, \alpha)\mapsto T_\alpha''(x)$ and $(x, \alpha)\mapsto T_\alpha'(x)$ are continuous  on $I\times[\alpha_0, \alpha_1]$.
\end{itemize}

Notice that the elements of $\F$ are parametrised according to the tangency near $0$. 
Now, we describe the randomising dynamics. 
Let $\eta$ be the normalised Lebesgue measure on $[\alpha_0, \alpha_1]$.  Let $\Omega=[\alpha_0, \alpha_1]^{\ZZ}$ and $\PP=\eta^{\ZZ}$. Then the shift map $\sigma:\Omega \to \Omega$ preservers $\PP$, i.e. $\sigma_\ast \PP=\PP$. For $\omega\in \Omega$, $\omega=\dots \omega_{-1}, \omega_0, \omega_1, \dots$ let $\alpha(\omega)=\omega_0\in [\alpha_0, \alpha_1]$. 
The \emph{random map} is formed by random compositions of maps $T_{\alpha(\omega)}:I\to I$ from $\F$, where the compositions are defined as  $T_{\omega}^n(x)=T_{\alpha(\sigma^{n-1}(\omega))}\circ\dots \circ T_{\alpha(\omega)}(x)$. Below we use more shorter notation  $T_\omega^n=T_{\omega_{n-1}}\circ\dots\circ T_{\omega_0}(x)$. 
We are interested in studying the statistical properties of equivariant families of measures i.e. families of measures $\{\mu_{\omega}\}_{\omega\in \Omega}$ such that $(T_{\omega})_\ast\mu_{\omega}=\mu_{\sigma\omega}$. Let $\mu$ be a probability measure on $I\times \Omega$ such that $\mu(A)=\int_\Omega\mu_\omega(A)d\PP(\omega)$ for $A\subset I\times \Omega$. We say that the system $\{ f_\omega, \mu_\omega\}_{\omega \in \Omega}$ (or simply $\{\mu_\omega\}_\omega$) is \emph{mixing} if  
for all $\varphi, \psi \in L^2(\mu)$, 
$$
\lim_{n\to \infty}\left|\int_\Omega\int_0^1\varphi_{\sigma^n\omega}\circ f^n_\omega\cdot\psi_\omega d\mu_\omega d\PP -\int_\Omega\int_0^1\varphi_{\omega}d\mu_\omega d\PP\int_\Omega\int_0^1\psi_{\omega}d\mu_\omega d\PP\right| =0.
$$
Further, future and past correlations are defined as follows. Let 
$\varphi, \psi :I\to \RR$ be two observables on $I$. Then we define \emph{future correlations } as 
$$
Cor_{\mu}^f(\varphi, \psi):= 
\left| \int (\varphi \circ T^{n}_\omega)\psi d\mu_{\sigma^n\omega}-\int \varphi d\mu_{\sigma^n\omega}\int \psi d\mu_{\omega}\right|
$$
and \emph{past correlations} as 
$$
Cor_{\mu}^p(\varphi, \psi):= 
\left| \int (\varphi \circ T^{n}_{\sigma^{-n}\omega})\psi d\mu_\omega-\int \varphi d\mu_{\omega}\int \psi d\mu_{\sigma^{-n}\omega}\right|.
$$
\begin{theorem}\label{main}
Let $T_{\omega}$ be the random map described above.  Then for almost every 
$\omega \in \Omega$ there exists a family of absolutely continuous equivariant measures $\{\mu_\omega\}_\omega$ on $I$ which is mixing. 
Moreover, for every $\delta>0$ there exists a full measure subset $\Omega_0\subset \Omega$ and a random variable $C_\omega: \Omega\to \RR_+$ which is finite on $\Omega_0$ such that for any $\varphi\in L^\infty(I)$, $\psi \in C^\eta(I)$ there exists a constant $C_{\varphi, \psi}>0$ so that 
$$
Cor_{\mu}^f(\varphi, \psi)\le C_\omega C_{\varphi, \psi}n^{1-\frac{1}{\alpha_0} +\delta}\,\, \text{and} \,\, Cor_{\mu}^p(\varphi, \psi)\le C_\omega C_{\varphi, \psi}n^{1-\frac{1}{\alpha_0} +\delta} .
$$
Furthermore,  there exist constants $C>0$, $u'>0$ and $0< v' <1$ such that
$$
P\{ C_\omega >n\} \leq Ce^{-u'n^{v'}}.
$$
\end{theorem}

\begin{remark}
Notice that in the deterministic setting every mapping in the family $\F$ admits an absolutely continuous invariant probability measure, which is polynomially mixing  at the rate $n^{1-1/\alpha}$ if $T_\alpha(x)=x+c_\alpha x^{1+\alpha}(1+f_\alpha(x))$ (see \cite{Y99},  \cite{G06}).  In the random setting the  upper bounds we give are arbitrarily close to the sharp decay rates of the fastest mixing system in the family. 
Since the result holds for almost every $\omega\in\Omega$, and in principle there can be arbitrarily long compositions of systems in $T^n_\omega$ whose mixing rates are slower than that of $T_{\alpha_0}$ it is not expected that the mixing rate of the random system will be the same as the mixing rate of the fastest mixing system in the family $\F$ and $C_\omega$ integrable at the same time.  
\end{remark}
\begin{remark}We also remark that we are choosing the family $\F$ so that all the maps in it share the common neutral fixed point $0$.  If we choose the family by allowing different maps having distinct neutral fixed points i.e. $T_\alpha (p(\alpha))=p(\alpha)$, $T_\alpha' (p(\alpha))=1$ and  $p(\alpha)\neq 0$  for  a positive (with respect to $\nu$) measure set of parameters $\alpha\in[\alpha_0, \alpha_1]$   and expanding elsewhere, then the resulting random map is expanding on average. Whence one can apply spectral  techniques as in \cite{Bu} on the Banach space of quasi-H\"older functions from \cite{Kel85} or \cite{Be00} and obtain exponential decay rates. Such systems are out of context in our setting since we are after systems with only polynomial decay of correlations.  
\end{remark}

To prove the theorem we construct a random induced map (or Random Young Tower)  for $T_{\omega}$  with the properties described in \cite{BBR}. Below we briefly recall the definition of induced map. 

Let  $m$ denote the Lebesgue measure on $I$ and  $\Lambda\subset I$ be a measurable subset. We say $T_{\omega}$ admits a Random Young Tower with the base $\Lambda$ if for almost every $\omega\in\Omega$  there exists a countable partition $\{\Lambda_j(\omega)\}_j$ of $\Lambda$ and  a return time function $R_\omega:\Lambda\to \mathbb N$ that is constant on each $\Lambda_j(\omega)$ such that 
\begin{itemize}
\item[(P1)]  $ \quad$ for each $\Lambda_j(\omega)$ the induced map $T^{R_\omega}_\omega|_{\Lambda_j(\omega)}\to \Lambda$ is a diffeomorphism and there exists a constant $\beta> 1$ such that $(T^{R_\omega}_\omega)'>{\beta}$.
\item[(P2)] $ \quad$ There exists $\mathcal D>0$ such that for all $\Lambda_j(\omega)$ and $x, y\in \Lambda_j(\omega)$ 
$$
\left|\frac{(T^{R_\omega}_\omega)'x}{(T^{R_\omega}_\omega)'y}-1\right|\le \mathcal D\beta^{-s(T^{R_\omega}_\omega(x), T^{R_\omega}_\omega(y))},
$$
where $s(x, y)$ is the smallest $n$ such that $(T^{R_\omega}_\omega)^nx$ and $(T^{R_\omega}_\omega)^ny$ lie in distinct elements.
\item[(P3)] $ \quad$  There exists $M>0$ such that  
$$\sum_{n}m\{x\in\Lambda\mid R_\omega(x)>n\}\le M \text{ for all } \omega\in\Omega.$$ 
There exist constants $C, u, v>0$,  $a>1$, $b\ge 0$, a full measure subset $\Omega_1\subset \Omega$,   and a random variable $n_1:\Omega_1\to \mathbb N$ so that
 \begin{equation}
\begin{cases} \label{tail}
m\{x\in\Lambda\mid R_\omega(x)>n\}\le C\frac{(\log n)^b}{n^a}, \,\,\text{whenever} \,\,n\ge n_1(\omega),\\
\PP\{n_1(\omega)>n\} \le C e^{-un^v},
\end{cases}
\end{equation}
\begin{equation}\label{tailann}
\int m\{x\in\Lambda |\, R_\omega=n\}d\PP(\omega)\le C\frac{(\log n)^{b}}{n^{a+1}}.
\end{equation}
\item[(P4)] $ \quad$ There are $N\in\mathbb N$ and 
$\{t_i\in\mathbb Z_+\mid i=1, 2, ..., N\}$   such that g.c.d.$\{t_i\}=1$ and $\epsilon_i >0$ so that
for almost every $\omega \in \Omega$ and $i = 1, 2, \dots N$ we have
$m\{x\in\Lambda\mid R_\omega(x)=t_i\}>\epsilon_i$.
\end{itemize}
Under the above assumptions it is proven in \cite{BBR} that there exists a family of absolutely continuous equivariant measures \cite[Theorem 4.1]{BBR}, which is mixing and the mixing rates have upper bound of the form  ${n^{1+\delta-a}}$ for any $\delta>0$  \cite[Theorem 4.2]{BBR}. Therefore to prove Theorem  \ref{main}  it is sufficient to construct an induced map $T^{R_\omega}_\omega$ with the properties 
(P1)-(P4), which is carried out in  the next section.

\section{Inducing scheme}\label{induce}
Here we will construct a uniformly expanding full branch induced random map on  $\Lambda=(0, 1]$ for  every $\omega \in \Omega$.  Let 
$X_0(\omega)=1$,  $X_1(\omega)=x(\omega_0)=x_{\alpha(\omega)}$ and 
$$X_n(\omega)=(T_\omega|_{[0, x(\omega_0))})^{-1}X_{n-1}(\sigma\omega) \text{ for }  n\ge 2.$$ 
Let $I_n(\omega)=(X_n(\omega), X_{n-1}(\omega)]$. Then by definition $T_\omega (I_n(\omega))=I_{n-1}(\sigma\omega)$. By induction we have 
$$ 
I_n(\omega) \xrightarrow{ \,\,T_{\omega \,\,\,}} I_{n-1}(\sigma\omega)\xrightarrow{\,T_{\sigma\omega \,\,}}\cdots I_1(\sigma^{n-1}\omega)\xrightarrow{T_{\sigma^{n-1}\omega}}\Lambda.
$$
Hence, every interval $I_n(\omega)$  first is mapped onto $I_1(\omega)$ and then is mapped onto $\Lambda$ by the next iterate of $T_\omega$.
Define a return time $R_\omega:(0, 1]\to \NN$ by setting $R_\omega|_{(X_{n}(\omega), X_{n-1}(\omega)]}=n$. Then the induced full branch map 
$T^{R_\omega}_\omega: (0, 1]\to (0, 1]$  defined as $T^{R_\omega}_\omega|_{I_n(\omega)}=T_\omega^n$, for  $n\ge 1$. By assumptions (A1) and (A2) there exists $\beta>1$ such that $T_{(\omega)}^{R_\omega}>\beta$ for all $\omega\in \Omega$.  In fact, we can choose 
\begin{equation}\label{xbeta}
\beta=\min_{\omega_0\in [\alpha_0, \alpha_1]}\min_{x\in  [x(\omega_0) ,1]}|T_{\omega_0}'(x)|.
\end{equation}
This proves (P1). By (A1) all the maps in $\F$ have two full branches with $x_\alpha<1$. Hence, the interval where $R_\omega=1$ has strictly positive length and thus (P4) is obviously satisfied. 

 To prove the remaining properties we use the  following  proposition, which is proved in section \ref{pof:key}. 
 \begin{proposition}\label{prop:key}
$1)$ For every  $\omega\in \Omega$ the sequence 
$\{X_n(\omega)\}_n$ is decreasing  and \\ $\lim_{n\to \infty}X_n(\omega)=~0$ . Moreover, there exists a constant $C_0>0$ such that for all $\omega\in \Omega$ 
\begin{equation}\label{estm:twosided}
\frac{1}{C_0n^{1/\alpha_0}}\le X_{n}(\omega) \le \frac{C_0}{n^{1/\alpha_1}}.
\end{equation}
 $2)$ There exists $C,u>0$, $v\in(0, 1)$ and a random variable $n_1:\Omega \to \NN$ which is finite for $\PP$-almost every $\omega\in \Omega$ such that  
\begin{align}
\label{eq:tailofn1} \PP\{\omega\mid n_1(\omega) >n\}\le C e^{-un^v},\\
\label{eq:tailofXn}X_{n}(\omega)\le C n^{-1/\alpha_0}(\log n)^{1/\alpha_0}\quad \forall n\ge n_1, \\ 
\label{eq:anntail} \int (X_{n-1}(\omega)-X_{n}(\omega))d\PP(\omega) \le C n^{-1-1/\alpha_0}(\log n)^{1/\alpha_0}.\end{align}
\end{proposition}

Now we will prove (P3).  For every $\omega\in\Omega$ by definition of $R_\omega$ and inequality \eqref{estm:twosided}  we have 
$$
m\{R_\omega >n\}=X_n(\omega)-x(\omega_0)\le \frac{1}{\beta}X_n(\sigma\omega)\le C_0 n^{-1/\alpha_1}.
$$
Since $\alpha_1<1$ we have $\sum n^{-1/\alpha_1}<+\infty$ and hence, there exists  $M>0$ such that 
$$
\sum_{n\ge 1}m\{R_\omega >n\}\le M.
$$

Inequalities  \eqref{eq:tailofn1} and  \eqref{eq:tailofXn} in Proposition \ref{prop:key} directly imply the inequalities \eqref{tail} in (P3), and \eqref{eq:anntail} implies inequality \eqref{tailann} in (P3). It remains to show distortion estimates (P2) for the induced map. Our proof is based on Koebe principle. Recall that the Schwarzian derivative of a $C^3$ diffeomorphism $g$ is  defined as 
$$
Sg(x)=\frac{g'''(x)}{g'(x)}-\frac{3}{2}\left(\frac{g''(x)}{g'(x)}\right)^2.
$$ 
It can be easily checked that if $f$ and $g$ are two maps  such that $f'\ge 0$, $Sf<0$ and $Sg\le 0$, then 
$S(g\circ f)=(Sg)\circ f\cdot f'+ Sf<0$ i.e.  the composition $g\circ f$ has negative Schwarzian derivative. We will use this observation in the proof of  Lemma \ref{kdist}.

Let $J\subset J'$ be two intervals and let $\tau>0$. $J'$ is called a $\tau$-scaled neighbourhood of $J$ if both components of $J'\setminus J$ have length at least $\tau |J|,$ where $|J|$ denotes the length of $J$. The Koebe principle  \cite[Chapter IV, Theorem 1.2]{dMvS} states that, if $g$ is a diffeomorphism onto its image with $Sg<0$ and $J\subset J'$ are two intervals  such that $g(J')$ contains $\tau$-scaled neighbourhood of $g(J)$ then there exists $\hat K(\tau)$ such that for any $x, y\in J$
\begin{equation}\label{Ktau}
\left|\frac{g'(x)}{g'(y)}-1\right| \le \hat K(\tau)\frac{|x-y|}{|J|}.
\end{equation}
By applying  the mean value theorem twice first in $J$ and then in  $(x, y)\subset J$ for any $x, y\in J$ we obtain 
$$
\frac{|g(x)-g(y)|}{|g(J)|}=\frac{|g'(v)|}{|g'(u)|}\frac{|x-y|}{|J|}
$$
for  some $u\in J$, $v\in (x, y)$.   Now inequality \eqref{Ktau} implies that $|g'(v)|/|g'(u)| \ge (1+\hat K(\tau))^{-1}$. Thus 
\begin{equation}\label{koebe}
\left|\frac{g'(x)}{g'(y)}-1\right| \le K(\tau)\frac{|g(x)-g(y)|}{|g(J)|},
\end{equation}
for $K(\tau)=(1+\hat K(\tau))\hat K(\tau)$.

Recall that by (A4) the left branch of $T_\omega$ has negative Schwarzian derivative for all $\omega\in\Omega$.  This fact will be used in the proof of the following  lemma. 
\begin{lemma}\label{kdist}

There exists $K>0$ such that for all $\omega\in \Omega$,  $n\in \NN$ and for $x, y\in I_{n}(\omega)$ 
$$
\left|\frac{(T^n_\omega)'(x)}{(T^n_\omega)'(y)}-1\right| \le K |T^n_\omega(x)-T^n_\omega(y)|.
$$
\end{lemma}

\begin{proof} Notice,  that  $M=\max_{\omega_0\in[\alpha_0, \alpha_1]}\max_{x\in I_1( \omega)}|T''_\omega(x)|<+\infty$  by (A5). Also, recall that  $T_\alpha'|_{I_\omega}>\beta>1$ for any $T_\alpha\in \F$. Thus for $n=1$, we have 
$$
\left|\frac{(T_\omega)'(x)}{(T_\omega)'(y)}-1\right| \le \frac{1}{\beta}|(T_\omega)'(x)-(T_\omega)'(y)|\le \frac{M}{\beta^2}|T_\omega(x)-T_\omega(y)|.
$$
For $n\ge 2$ we use Koebe principle mentioned above.   Set $J=[X_{n}(\omega), X_{n-1}(\omega)]$ and $J'=[X_{n+1}(\omega), 2]$. We  first extend $T_{\omega_{n-2}}$, $\dots$, $T_{\omega_0}$ to $(0, +\infty)$ analytically, keeping the Schwarzian derivative non-positive\footnote{Such extensions can be constructed easily. For example, for $f\in \F$ it is sufficient to take $\tilde f(x)=a(x-x_\alpha)^4+b(x-x_\alpha)^3+c(x-x_\alpha)^2+d(x-x_\alpha)+1$ with $a<bc/d$, where $a,b, c$ are the Taylor coefficients of $f$ at $x=x_\alpha$.}. Let $g=T_{\omega_{n-2}}\circ \dots \circ T_{\omega_0}$. Then, $g$ has negative Schwarzian derivative. We will show that $g(J')$ contains  $\tau$ scaled neighbourhood of $g(J)$ for some $\tau>0$, which is independent of $\omega$.  Since $g(X_{n}(\omega))=X_1(\sigma^{n-1}\omega)$ and $g(X_{n+1}(\omega))=X_2(\sigma^{n-1}\omega)$.  It is sufficient to show that $X_{1}(\omega)-X_{2}(\omega)$ is bounded below by a constant independent of $\omega$. By definition of $X_n$ we have 
$$
|X_{1}(\omega)-X_{2}(\omega)|= |T_\omega^{-1}(1)-T_\omega^{-1}\circ T_{\sigma(\omega)}^{-1}(1)| \ge \frac{1}{\beta'}|1-T_{\sigma(\omega)}^{-1}(1)|\ge\kappa>0,
$$
where $\beta'=\min\{T_\omega'(x)\mid (x, \omega_0) \in [ \tilde X, 1]\times[\alpha_0, \alpha_1] \}> 1$ with $\tilde X=\min_{\omega} X_2(\omega)$ and 
$\kappa=\beta' (1-\min_{\alpha}x_\alpha)>0$ by  (A1).
Thus, using the fact $|g(J)|>1- \max_\alpha x_\alpha >0$  from \eqref{koebe} we obtain 
\begin{equation*}
\left|\frac{g'(x)}{g'(y)}-1\right| \le K{|g(x)-g(y)|}.
\end{equation*}
with $K={K(\tau)}/{1- \max_\alpha x_\alpha}$ which finished the proof. 
\end{proof}

\begin{lemma}\label{d1}
There exists a constant $C>0$ independent of $\omega$ such that for all $\omega\in \Omega$ and for any $x, y\in I_n(\omega)$
$$
\left|\log\frac{(T_\omega^{R_\omega})'(x)}{(T_\omega^{R_\omega})'(y)}\right| \le C|T_\omega^{R_\omega}(x)-T_\omega^{R_\omega}(y)|.
$$
\end{lemma}
\begin{proof} From now on we suppress the $\omega$ in $R_\omega$, since no confusion arises. 
Note that $T^{R}_{\omega}(x)$ is the composition of the right branch of $T_{\sigma^{R-1}\omega}$ and $g$ i.e. $T^{R}_{\omega}(x)=T_{\sigma^{R-1}\omega}\circ g(x)$. Therefore, by  definition of $\beta$ in \eqref{xbeta} by  Lemma \ref{kdist} we have 
$$
\begin{aligned}
\log\left| \frac{(T^{R}_{\omega})'(x)}{(T^{R}_{\omega})'(y)}\right|\le K|T^R(x)-T^R(y)|+K|g(x)-g(y)| \le K(1+\frac{1}{\beta})|T^R(x)-T^R(y)|.
\end{aligned}
$$
\end{proof}
Now, we will prove (P2). Together with  an elementary inequality $|x-1| \le C|\log(x)|$ (for some $C>0$, whenever  $|\log x|$ is bounded above) Lemma \ref{d1} implies that for  any $x, y\in I_n(\omega)$ we have 
$$
\left|\frac{(T^{R}_\omega)'x}{(T^{R}_\omega)'y}-1\right|\le  D(K, \beta)|T^R(x)-T^R(y)| \le \mathcal D\beta^{-s(T_\omega^R(x), T_\omega^R(y))},
$$
where $\mathcal D=D(K, \beta)$ is a constant that depends only on $K$ and $\beta$ and the last inequality follows from the observation:  if $x, y\in (0, 1]$ are such that $s(x, y)=n$ then $|x-y|\le \beta^{-n}$. Indeed, by definition $(T_\omega^{R})^{i}(x)$  and $(T_\omega^{R})^{i}(y)$ belong to the same  element of the  partition $\{I_{k}(\omega)\}$ for  all $i=0, ..., n-1$. Thus by the mean value theorem 
$$|x-y|= |[(T_\omega^{R})^{n}]'(\xi)|^{-1}|(T_\omega^{R})^{n}(x) - (T_\omega^{R})^{n}(y)|\le \beta^{-n}.$$ 

\section{Proof of Proposition \ref{prop:key}}\label{pof:key}
 We start by proving  an  auxiliary lemma, which is used in the proof.  
\begin{lemma}\label{lem:expec}
For any $k\in \NN$, $c\ge 1$ and $t>0$ we have 
$$
E_{\PP}[e^{-(c\alpha(\sigma^{k}\omega)-\alpha_0)t}]=\frac{1}{\alpha_1-\alpha_0}\frac{e^{\alpha_0t(1-c)}}{ct}(1-e^{-ct(\alpha_1-\alpha_0)}).
$$
\end{lemma}

\begin{proof}
Since $\sigma$ preserves $\PP$ we have 
\begin{align*}
&E_{\PP}[e^{-(c\alpha(\sigma^{k}\omega)-\alpha_0)t}]=E_{\PP}[e^{-(c\alpha(\omega)-\alpha_0)t}] \\
&=\frac{1}{\alpha_1-\alpha_0}\int_{\alpha_0}^{\alpha_1}e^{-(cx-\alpha_0)t}dx
%=\frac{1}{\alpha_1-\alpha_0}\frac{1}{ct}(e^{-c\alpha_0t+\alpha_0t}-e^{-c\alpha_1t+\alpha_0t})\\
=\frac{1}{\alpha_1-\alpha_0}\frac{e^{\alpha_0t(1-c)}}{ct}(1-e^{-ct(\alpha_1-\alpha_0)}).
\end{align*}
\end{proof}

\begin{proof} Now we are ready to prove Proposition 1. First we prove item 1).
The first two assertions are  obvious, since $T'(x)>1$ for $x>0$ and $x=0$ is the unique fixed point in $[0, 1/2]$.  Since all the maps in $\F$ are uniformly expanding except at $0$,  there exists $n_0\in \NN$ independent of $\omega$ such that $X_n(\omega)\in (0, \eps_0)$ for all $n\ge n_0$. Thus, it is sufficient to prove inequality \eqref{estm:twosided} for any $n\ge n_0$.  We now define a sequence 
$\{Z_n\}_n$ which bounds $X_n(\omega)$ from below and has desired asymptotic.   Let $K_0=[0, \eps_0]\times [\alpha_0, \alpha_1]$ and  $C_1=\max_{(x, \alpha)\in K_0}c_\alpha(1+f_\alpha(x))$. Set $G(x)=x(1+C_1x^{\alpha_0})$.  Define $\{Z_n\}_{n\ge n_0}$ as follows:  $Z_{n_0}=\min_{\omega\in \Omega}X_{n_0}(\omega)$ and let $Z_{n}=(G|_{[0, \eps_0]})^{-1}(Z_{n-1})$  for $n >n_0$. Since $G(x)\ge T_{\alpha(\omega)}(x)$ for any $x\in [0, \eps_0]$ and for any $\omega\in \Omega$, one can easily verify by induction that $Z_n\le X_n(\omega)$ for $n\ge n_0$. Finally note that $Z_n\sim n^{-1/\alpha_0}$ \cite{G06}. Defining $C_1'=\min_{(x, \alpha)\in K_0}c_\alpha(1+f_\alpha(x))$,  $G'(x)=x(1+C_1'x^{\alpha_1})$, $Z_{n_0}'=\max_{\omega\in \Omega}X_{n_0}(\omega)$ and  $Z_{n}'=(G'|_{[0, \eps_0]})^{-1}(Z_{n-1}')$  for $n >n_0$ we obtain a sequence $\{Z_n'\}$ such that $X_n(\omega)\le Z_n'$ and $Z_n'\sim  n^{-1/\alpha_1}$.
This finishes the proof. 
\end{proof}

Item 2) is proved below.  Note that by the choice of $n_0$ for any $n\ge n_0$ we have 
\begin{equation}\label{eq:Xnreprs}
X_{n}(\sigma\omega)=X_{n+1}(\omega)[1+c_{\alpha(\omega)}X_{n+1}(\omega)^{\alpha(\omega)}(1+f_{\alpha(\omega)}\circ X_{n+1}(\omega))].
\end{equation}
The latter equality together with the standard estimate $(1+x)^{-a}\le 1-ax+\frac{a(a+1)}{2}x^2$ for $x, a >0$ implies that 
$$
\frac{1}{X_{n+1}(\omega)^{\alpha_0}}-\frac{1}{X_{n}(\sigma\omega)^{\alpha_0}} \ge C_1\alpha_0 X_{n+1}(\omega)^{\alpha(\omega)-\alpha_0}- C_2 X_{n+1}(\omega)^{2\alpha(\omega)-\alpha_0},
$$
where , $C_2=\frac{\alpha_0(\alpha_0+1)}{2}\min_{(\alpha, x)\in K_0}[c_\alpha(1+f_{\alpha}(x))]^2$. Hence,
$$
\frac{1}{X_{n}(\omega)^{\alpha_0}}\ge\frac{1}{x_{\alpha(\omega)}^{\alpha_0}}+C_1\alpha_0\sum_{k=2}^{n} X_{k}(\omega)^{\alpha(\sigma^{n-k}\omega)-\alpha_0}- C_2 \sum_{k=2}^{n}X_{k}(\omega)^{2\alpha(\sigma^{n-k}\omega)-\alpha_0},
$$
Notice that we can take $C_1$ and $C_2$ independent of $\omega$. Therefore, by inequality \eqref{estm:twosided}  we have 
\begin{equation}\label{Xnw}
\frac{1}{X_{n}(\omega)^{\alpha_0}}\ge
1+C_3\sum_{k=2}^{n} (k^{1/\alpha_0})^{\alpha_0-\alpha(\sigma^{n-k}\omega)}- C_2 \sum_{k=2}^{n}(k^{1/\alpha_1})^{(-2\alpha(\sigma^{n-k}\omega)+\alpha_0)},
\end{equation}
First we will show that  the  right hand side of the latter inequality  on average behaves like $n^{-1}\log n$ as $n$ goes to infinity. 
We set 
$$a_k:= (k^{1/\alpha_0})^{\alpha_0-\alpha(\sigma^{n-k}\omega)}, \quad b_k=(k^{1/\alpha_1})^{-2\alpha(\sigma^{n-k}\omega)+\alpha_0}$$ 
and 
$$S_n=\sum_{k=2}^{n}C_3a_k-C_2b_k.$$

\begin{lemma} \label{4} There exists $C_4>0$ such that $\displaystyle\lim_{n\to \infty}\frac{\log n }{n}E_{\PP}(S_n)=C_4.$
\end{lemma}
\begin{proof}
Applying the above lemma to $E_{\PP}(e^{\log a_k})$ with $c=1$, $u=\log k^{1/\alpha_1}$ and using the fact $\sum_{k\le n}\frac{1}{\log k}\sim \frac{n}{\log n}$ we obtain 
$$
\sum_{k=2}^{n}E_{\PP}(a_k)=\frac{\alpha_0}{\alpha_1-\alpha_0}\sum_{k=2}^n\frac{1}{\log k}(1-k^{-\frac{\alpha_1-\alpha_0}{\alpha_0}})= \frac{\alpha_0}{\alpha_1-\alpha_0}\frac{n}{\log n}+O(n^{1-\frac{\alpha_1-\alpha_0}{\alpha_0}}(\log n)^{-1})
$$
and hence, 
\begin{equation}\label{Ea_k}
\frac{\log n }{n}\sum_{k=2}^{n}E_{\PP}(a_k) = \frac{\alpha_0}{\alpha_1-\alpha_0}+O(n^{-\frac{\alpha_1-\alpha_0}{\alpha_0}}).
\end{equation}
Similarly, applying Lemma \ref{lem:expec} to $E_{\PP}(b_k)$ with $c=2$ and $t=\log k^{1/\alpha_1}$,  we obtain 
$$
\sum_{k=2}^{n}E_{\PP}(b_k):=\frac{\alpha_1}{2(\alpha_1-\alpha_0)}\sum_{k=2}^{n}\frac{1}{\log k}(k^{-\frac{\alpha_0}{\alpha_1}} - k^{\frac{\alpha_0}{\alpha_1}-2})=\frac{\alpha_1}{2(\alpha_1-\alpha_0)}\frac{n^{1-{\alpha_0}/{\alpha_1}}}{\log n}+ o(n).
$$
and hence, 
\begin{equation}\label{Eb_k}
\lim_{n\to\infty}\frac{\log n }{n}\sum_{k=2}^{n}E_{\PP}(b_k)=\lim_{n\to\infty} n^{-\alpha_0/\alpha_1}=0.
\end{equation}
Combining \eqref{Ea_k} and \eqref{Eb_k} implies 
$$
\lim_{n\to \infty}\frac{\log n }{n}E_{\PP}(S_n)=\lim_{n\to \infty}\frac{\log n }{n}\sum_{k=2}^{n}E_{\PP}(C_3a_k-C_2b_k) =C_4,
$$
where $C_4= {C_3\alpha_0}/(\alpha_1-\alpha_0)$. 
\end{proof}

Now we construct a random variable $n_1:\Omega\to \NN$ as in item 2) of Proposition \ref{prop:key}. 
Lemma \ref{4} implies that  there exists $N$ independent of $\omega$ such that 
\begin{equation}\label{Eak-Ebk}
\frac{C_4}{2}\le\frac{\log n }{n}E_{\PP}(S_n) \le\frac{3C_4}{2}
\end{equation}
for all $n\ge N$.
On the other hand, by \cite[Theorem 1]{Hoe}, there exists $C>0$ such that for every $t>0$ and $n\in \NN$ we have 
$$
\PP\left\{\frac{\log n}{n}|S_{n+1}-E_{\PP}(S_{n+1})|< t\right\}\le e^{-\frac{Cnt^2}{(\log n)^2}} .
$$
Thus, by letting $C_5=CC_4^2/16$ we obtain 
\begin{equation}\label{pp}
\PP\left\{\frac{\log n}{n}S_{n+1}< \frac{C_4}{4}\right\}\le 
\PP\left\{\frac{\log n}{n}(S_{n+1}-E_{\PP}S_{n+1})<  -\frac{C_4}{4}\right\}\le e^{-\frac{C_5n}{(\log n)^2}}.
\end{equation}
Define 
$$n_1(\omega)=\inf\{n\ge N\mid \forall k\ge n, \frac{\log k}{k}S_k\ge \frac{C_4}{4}\}.$$ 
Inequality \eqref{pp}  implies that 
\begin{equation*}\label{eq:n1final}
\PP\{n_1(\omega)>n\}\le \sum_{k=n}^{\infty}e^{-\frac{C_5k}{(\log k)^2}}\le C_6\sum_{k=n}^{\infty}e^{-uk^v}\le Ce^{-un^v}
\end{equation*}
for some $C>0$, $u>0$ and $v\in(0, 1)$ which proves inequality \eqref{eq:tailofn1}.

For any $n\ge n_1$ by \eqref{Xnw} we have 
$$
X_n(\omega)^{\alpha_0} \le \frac{\log n}{n}\frac{4}{C_4}.
$$
Hence, for some positive $C>0$ we have 
$$
X_n(\omega)\le C \left (\frac{\log n}{n}\right)^{1/\alpha_0}.
$$
This finishes the proof of  \eqref{eq:tailofXn}. It remains to prove \eqref{eq:anntail}.  Recall that there exists $n_0$ which depends only on $\eps_0$ in (A3) such that \eqref{eq:Xnreprs} holds for all $n\ge n_0$. Thus, recalling that  $\sigma$ preserves $\PP$ we  have 
$$
\begin{aligned}
\int m\{R_\omega=n\}d\PP(\omega)\le \frac{1}{\beta}\int (X_{n-1}(\sigma\omega)-X_n(\sigma\omega))d\PP(\omega)= \\
  \frac{1}{\beta}\int (X_{n-1}(\sigma\omega)-X_n(\omega))d\PP(\omega)=\\
\int_{\{n_1(\omega)>n\}} (X_{n-1}(\sigma\omega)-X_n(\omega))d\PP(\omega) +\int_{\{n_1(\omega)\le n\}} (X_{n-1}(\sigma\omega)-X_n(\omega))d\PP(\omega)
\\
\le Ce^{-un^v}+\int_{\{n_1(\omega)\le n\}} c_{\alpha(\omega)}X_{n}(\omega)^{\alpha(\omega)+1}(1+f_{\alpha(\omega)}\circ X_{n}(\omega)d\PP(\omega)\\
\le Ce^{-un^v}+ C\int\left(\frac{\log n}{n}\right)^{(\alpha(\omega)+1)/\alpha_0}d\PP(\omega) \le C\left(\frac{\log n}{n}\right)^{(\alpha_0+1)/\alpha_0}.
\end{aligned}
$$
This finishes the proof for all $n\ge n_0$. For $n<n_0$ the assertion follows by increasing the constant $C$ if necessary.

\begin{acknowledgement}
This research was supported by The Leverhulme Trust through the research grant RPG-2015-346. The author would like to thank Wael Bahsoun for useful discussions during the preparation of the paper. 
\end{acknowledgement}

%%%%%%%%%%%%%%%%%%%%%%%% referenc.tex %%%%%%%%%%%%%%%%%%%%%%%%%%%%%%
% sample references
% %
% Use this file as a template for your own input.
%
%%%%%%%%%%%%%%%%%%%%%%%% Springer-Verlag %%%%%%%%%%%%%%%%%%%%%%%%%%
%
% BibTeX users please use
% \bibliographystyle{}
% \bibliography{}
%

% Use the following (APS) syntax and markup for your references if 
% the subject of your book is from the field 
%% "Mathematics, Physics, Statistics, Computer Science"
%%
%% Online Document
%\bibitem{phys-online} J. Dod, in \textit{The Dictionary of Substances and Their Effects}, Royal Society of Chemistry. (Available via DIALOG, 1999), 
%\url{http://www.rsc.org/dose/title of subordinate document. Cited 15 Jan 1999}
%%
%% Monograph
%\bibitem{phys-mono} H. Ibach, H. L\"uth, \textit{Solid-State Physics}, 2nd edn. (Springer, New York, 1996), pp. 45-56 
%%
%% Journal article
%\bibitem{phys-journal} S. Preuss, A. Demchuk Jr., M. Stuke, Appl. Phys. A \textbf{61}
%%
%% Journal article by DOI
%\bibitem{phys-DOI} M.K. Slifka, J.L. Whitton, J. Mol. Med., doi: 10.1007/s001090000086
%%
%% Contribution 
%\bibitem{phys-contrib} S.E. Smith, in \textit{Neuromuscular Junction}, ed. by E. Zaimis. Handbook of Experimental Pharmacology, vol 42 (Springer, Heidelberg, 1976), p. 593
%%
%\bigskip
%%

\end{document}